\documentclass[a4paper]{amsart}
\usepackage{graphicx}
\usepackage{bm}
\usepackage{color}
\usepackage{amssymb, amsmath, amsthm}
\usepackage{bm}
\theoremstyle{plain}
\usepackage{graphicx}
\usepackage{natbib}
\usepackage{url}
\usepackage[plainpages=false,pdfpagelabels,breaklinks,hyperfootnotes=false]{hyperref}

\newtheorem{lemma}{Lemma}[section]
\newtheorem{theorem}[lemma]{Theorem}
\newtheorem{cor}[lemma]{Corollary}
\newtheorem{prop}[lemma]{Proposition}
\newtheorem{exam}[lemma]{\normalfont \scshape
 Example}

\newcommand{\R}{\mathbb{R}}
\newcommand{\N}{\mathbb{N}}
\newcommand{\norm}[1]{\left\Vert#1\right\Vert}
\newcommand{\abs}[1]{\left\vert#1\right\vert}
\newcommand{\set}[1]{\left\{#1\right\}}

\newcommand{\bfx}{\bm{x}}
\newcommand{\bfzero}{\bm{0}}

\newcommand{\bfone}{\bm{1}}
\newcommand{\bfa}{\bm{a}}
\newcommand{\bfb}{\bm{b}}
\newcommand{\bfc}{\bm{c}}
\newcommand{\bfd}{\bm{d}}
\newcommand{\bfe}{\bm{e}}
\newcommand{\bfQ}{\bm{Q}}
\newcommand{\bfR}{\bm{R}}

\newcommand{\bfU}{\bm{U}}
\newcommand{\bfu}{\bm{u}}
\newcommand{\bfv}{\bm{v}}
\newcommand{\bfV}{\bm{V}}

\newcommand{\bfX}{\bm{X}}
\newcommand{\bfY}{\bm{Y}}

\newcommand{\bfZ}{\bm{Z}}

\newcommand{\bfeta}{\bm{\eta}}

\allowdisplaybreaks[4]

\begin{document}

\title{The Multivariate Piecing-Together Approach Revisited}

\author{Stefan Aulbach}

\author{Michael Falk}

\author{Martin Hofmann}

\address{University of W\"{u}rzburg,
Institute of Mathematics,  Emil-Fischer-Str.~30, 97074~W\"{u}rz\-burg, Germany}

\email{stefan.aulbach@uni-wuerzburg.de\\
falk@mathematik.uni-wuerzburg.de\\ hofmann.martin@mathematik.uni-wuerzburg.de}



\begin{abstract}
The univariate Piecing-Together approach (PT) fits a univariate generalized Pareto
distribution (GPD) to the upper tail of a given distribution function in a continuous manner. A
multivariate extension was established by \citet{aulbf11}: The upper tail of a given copula
$C$ is cut off and replaced by a multivariate GPD-copula in a continuous manner, yielding a
new copula called a PT-copula. Then each margin of this PT-copula is transformed by a
given univariate distribution function. This provides a multivariate distribution function with
prescribed margins, whose copula is a GPD-copula that coincides in its central part with $C$.
In addition to \citet{aulbf11}, we achieve in the present paper an exact representation of
the PT-copula's upper tail, giving further insight into the multivariate PT approach. A variant
based on the empirical copula is also added. Furthermore our findings enable us to establish
a functional PT version as well.
\end{abstract}

\maketitle

\section{Introduction}

As shown by \citet{balh74} and \citet{pick75}, the upper tail of a univariate distribution
function $F$ can reasonably be approximated only by that of a \emph{generalized Pareto
distribution} (GPD), which leads to the Peaks-Over-Threshold (POT) approach: Set for a
univariate random variable $X$  with distribution function $F$
\begin{align*}
F^{[x_0]}(x)={\rm P}(X\le x\mid X> x_0)
=\frac{F(x)-F(x_0)}{1-F(x_0)},\qquad x\ge x_0,
\end{align*}
where we require $F(x_0)< 1$. The univariate POT is the
approximation of the upper tail of $F$ by that of a GPD
\begin{align*}
F(x)&=\{1-F(x_0)\}F^{[x_0]}(x)+F(x_0)\\
&\approx_{\mathrm{POT}}
\{1-F(x_0)\}Q_{\gamma,\mu,\sigma}(x)+F(x_0),\qquad x\ge x_0,
\end{align*}
where $\gamma$, $\mu$, $\sigma$ are shape, location and scale parameter of the GPD $Q$,
respectively. The family of univariate standardized GPD is given by
\begin{align*}
Q_{1,\alpha}(x)&=1-x^{-\alpha},\qquad x\ge 1,\\
Q_{2,\alpha}(x)&=1-(-x)^{\alpha},\qquad -1\le x\le 0,\\
Q_3(x)&=1-\exp(-x),\qquad x\ge 0,
\end{align*}
being the Pareto, beta and exponential GPD. Note that $Q_{2,1}(x)=1+x$, $-1\le x\le 0$, is
the uniform distribution function on $(-1,0)$. Multivariate GPD with these margins will play a
decisive role in what follows.

The preceding considerations lead to the univariate Piecing-Together approach (PT), by
which the underlying distribution function $F$ is replaced by
\begin{equation}\label{eqn:univariate_pt}
F_{x_0}^*(x)=\begin{cases} F(x),&x< x_0,\\
\{1-F(x_0)\}Q_{\gamma,\mu,\sigma}(x)+F(x_0),&x\ge x_0,
\end{cases}
\end{equation}
typically in a continuous manner. This approach aims at an  investigation of the upper end of
$F$ beyond observed data. Replacing $F$ in \eqref{eqn:univariate_pt} by the empirical
distribution function of the data provides in particular a semiparametric approach to the
estimation of high quantiles; see, e.g., \citet[Section 2.3]{reist07}.

A multivariate extension of the univariate PT approach was developed in \citet{aulbf11}
and,  for illustration, applied to operational loss data. This approach is based on the idea
that a multivariate distribution function $F$ can be decomposed into its copula $C$ and its
marginal distribution functions. The multivariate PT approach then consists of the two steps:
\begin{enumerate}
\item The upper tail of the given $d$-dimensional copula $C$ is cut off and substituted
    by the upper tail of a multivariate \emph{GPD-copula} in a continuous manner such
    that the result is again a copula, called a PT-copula.  Figure \ref{fig:pt-approach}
    illustrates the approach in the bivariate case: The copula $C$ is replaced in the upper
    right rectangle of the unit square by a GPD-copula $Q$; the lower part of $C$ is kept
    in the lower left rectangle, whereas the other two rectangles are needed for a
    continuous transition from $C$ to $Q$.

\item Univariate distribution functions $F^*_1, \dots , F^*_d$ are injected into the
    resulting copula.
\end{enumerate}
Taken as a whole, this approach provides a multivariate distribution function with prescribed
margins $F_i^*$, whose copula coincides in its lower or central part with $C$ and in its
upper tail with a GPD-copula.

\begin{figure}
\input{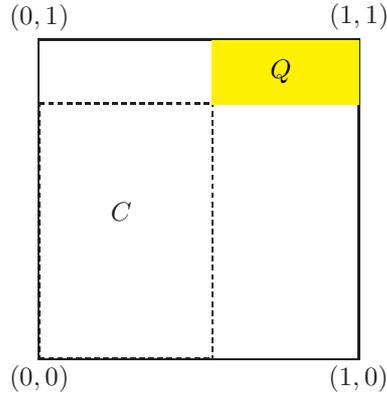}
\caption{The upper tail of a given  copula $C$ is cut off and
replaced by GPD-copula $Q$.}\label{fig:pt-approach}
\end{figure}

While in the paper by \citet{aulbf11} it was merely shown that the generated  PT-copula is
a GPD-copula, we achieve in the present paper an \emph{exact} characterization, yielding
further insight into the multivariate PT approach. A variant based on the empirical copula is
also added. Our findings enable us to establish a functional PT version as well.

The present paper is organized as follows.  In Section
\ref{sec:auxiliary_results} we compile basic definitions,
auxiliary results and tools. The multivariate PT result by
\citet{aulbf11} will be revisited and greatly improved in Section
\ref{sec:multivariate_piecing_together}. In Section
\ref{sec:pt_a_functional_version} we will extend the multivariate
PT approach to functional data.

\section{Auxiliary Results and Tools}\label{sec:auxiliary_results}

In this section we compile several auxiliary results and tools
from multivariate extreme value theory (EVT). Precisely, we
characterize in Proposition
\ref{prop:representation_of_copula_in_domain_of_attraction},
Corollary \ref{cor:expansion_of_copula} and Corollary
\ref{cor:characterization_of_domain_of_attraction_for_copula} the
max-domain of attraction of a multivariate distribution function
in terms of its copula. This implies an expansion of the lower
tail of a survival copula in Corollary \ref{cor:tail_copula}.
Lemma \ref{lem:characterization_of_multivariate_gpd} provides a
characterization of multivariate GPD in terms of random vectors.
For recent accounts of basic and advanced topics of EVT, we refer
to the monographs by \citet{dehaf06}, \citet{resn07,resn08} and
\citet{fahure10}, among others.

 Let $F$ be an arbitrary $d$-dimensional distribution function that is in the \emph{domain of attraction}
 of a $d$-dimensional extreme value distribution (EVD) $G$ (denoted by $F\in\mathcal D(G)$), i.e.,
 there exist norming constants $\bfa_n>\bfzero\in\R^d$, $\bfb_n\in\R^d$ such that
\begin{equation*}
F^n(\bfa_n\bfx+\bfb_n)\to_{n\to\infty}G(\bfx), \qquad \bfx\in\R^d,
\end{equation*}
where all operations on vectors are meant componentwise.  The distribution function $G$ is
max-stable, i.e., there exist norming constants $\bfc_n>\bfzero\in\R^d$, $\bfd_n\in\R^d$
with
\[
G^n(\bfc_n\bfx+\bfd_n)=G(\bfx),\qquad \bfx\in\R^d.
\]
The one-dimensional margins $G_i$ of $G$ are up to scale and
location parameters univariate EVD. With shape parameter
$\alpha>0$, the family of (univariate) standardized EVD is
\begin{align*}\label{eqn:univariate_EVD}
G_{1,\alpha}(x)&=\exp\left(-x^{-\alpha}\right),\qquad x>0,\nonumber\\
G_{2,\alpha}(x)&=\exp\left\{-(-x)^{\alpha}\right\},\qquad x\le 0,\\
G_3(x)&=\exp\left(-e^{-x}\right),\qquad x\in\R,\nonumber
\end{align*}
being the Fr\'{e}chet, (reverse) Weibull and Gumbel  EVD, respectively.

The following two results are taken from \citet{aulbf11}.

\begin{prop}\label{prop:representation_of_copula_in_domain_of_attraction}
A distribution function $F$ with copula $C_F$ satisfies $F\in\mathcal{D}(G)$ if, and only if,
this is true for the univariate margins of $F$ and if the expansion
\begin{equation}\label{eqn:expansion_of_copula}
C_F(\bfu)=1-\norm{\bfone-\bfu}_D+o(\norm{1-\bfu})
\end{equation}
holds uniformly for $\bfu\in[0,1]^d$, where $\norm\cdot_D$ is some $D$-norm.
\end{prop}

A $D$-norm $\norm\cdot_D$ on $\R^d$ is defined by
\[
\norm \bfx_D:={\rm E}\left\{\max_{1\le i\le d}(\abs{x_i}Z_i)\right\},\qquad \bfx\in\R^d,
\]
where $\bfZ=(Z_1,\dots,Z_d)$ is an arbitrary random vector which satisfies
$\bfZ\in[0,c]^d$ for some $c>0$ together with ${\rm E}(Z_i)=1$, $1\le i\le d$. In this case
$\bfZ$ is called \emph{generator} of $\norm\cdot_D$. Note that $\bfZ$ is not uniquely
determined.

For example, any random vector of the form
$\bfZ=2(U_1,\dots,U_d)$, with $(U_1,\dots,U_d)$ following an
arbitrary copula, can be utilized as a generator. This embeds the
set of copulas into the set of $D$-norms.

The index $D$ reflects the fact that for $(t_1,\dots,t_{d-1})\in[0,1]^{d-1}$ with $t_1 +
\dots + t_{m-1}\le 1$,
\[
D(t_1,\dots,t_{d-1}):=\norm{\left(t_1,\dots,t_{d-1},1-\sum_{i=1}^{d-1}t_i\right)}_D
\]
is the \emph{Pickands dependence function}, which provides another
way of representing a multivariate EVD $G$ with standard negative
exponential margins:
\[
G(\bfx)=\exp\left(-\norm{\bfx}_D\right)=
\exp\left(-\norm{\bfx}_1D\left(\frac{x_1}{\norm{\bfx}_1}\,\dots,\frac{x_{d-1}}{\norm{\bfx}_1}\right)\right),\quad
\]
for $\bfx\le\bfzero\in\R^d$, where $\norm{\bfx}_1=\abs{x_1}+\dots+\abs{x_d}$ is the
usual $p$-norm on $\R^d$ with $p=1$; for details we refer to \citet[Section
4.4]{fahure10}.

The following consequence of Proposition
\ref{prop:representation_of_copula_in_domain_of_attraction} is
obvious. This result is also already contained in \citet{aulbf11}.

\begin{cor}\label{cor:expansion_of_copula}
Let $F=C$ be a copula itself. Then $C\in\mathcal{D}(G)$ $\iff$
\eqref{eqn:expansion_of_copula} holds.
\end{cor}

The next result provides an expansion of the lower tail of the
\emph{survival copula}
\[
\bar C(u_1,\dots,u_d)={\rm P}(1-U_1\le u_1,\dots,1-U_d\le u_d),\qquad
\bfu\in[0,1]^d
\]
corresponding to any random vector $\bfU$, whose distribution is a copula $C$ with $
C\in\mathcal D(G)$. It will be used in the derivation of Proposition
\ref{prop:justification_of_multivariate_pt-approach}.

\begin{cor}\label{cor:tail_copula}
Let $(U_1,\dots,U_d)$ follow a copula $C\in\mathcal{D}(G)$, with corresponding $D$-norm
generated by the random vector $\bfZ=(Z_1,\dots,Z_d)$. Then for $\bfx\le\bfzero\in\R^d$
\[
\frac{{\rm P}(U_1>1+tx_1,\dots, U_d>1+tx_d)}t \to_{t\downarrow 0} {\rm E}\left\{\min_{1\le i\le d}(\abs{x_i}Z_i)\right\}=:\lambda(\bfx),
\]
where the function $\lambda$ is known as the \emph{tail copula}
(\citet{Kluppelberg_Kuhn_Peng_2006}).
\end{cor}

\begin{proof} First note that we have for arbitrary real numbers $a_1,\dots,a_d$ the equality
\[
\min(a_1,\dots,a_d)=\sum_{\emptyset\not= K\subset\set{1,\dots,d}}(-1)^{\abs{K}-1}\max(a_k:\,k\in K),
\]
which can be seen by induction. Denote by $\bfe_k$ the $k$-th unit vector in the Euclidean
space $\R^d$. The inclusion-exclusion theorem together with Corollary
\ref{cor:expansion_of_copula} then implies for fixed $\bfx\le\bfzero\in\R^d$ and arbitrary
$t>0$
 \begin{align*}
   &{\rm P}(U_1>1+tx_1,\dots, U_d>tx_d)\\
    &=1-{\rm P}\left(\bigcup_{i=1}^d\set{U_i\le 1+tx_i}\right)\\
    &=1- \sum_{\emptyset\not=K\subset\set{1,\dots,d}}(-1)^{\abs{K}-1}{\rm P}(U_k\le 1+t x_k,\,k\in K)\\
    &= 1- \sum_{\emptyset\not=K\subset\set{1,\dots,d}}(-1)^{\abs{K}-1}\left(1-t\norm{\sum_{k\in K}x_k\bfe_k}_D\right)+ o(t)\\
    &= t \sum_{\emptyset\not=K\subset\set{1,\dots,d}}(-1)^{\abs{K}-1} {\rm E}\left(\max_{k\in K}(\abs{x_k} Z_k)\right)+ o(t)\\
    &= t {\rm E}\left(\min_{1\le i\le d}(\abs{x_i}Z_i)\right) + o(t),
 \end{align*}
 which yields the assertion.
\end{proof}

A $d$-dimensional distribution function $Q$ is called
\emph{multivariate GPD} iff its upper tail equals $1+\ln(G)$,
precisely, iff there exists a $d$-di\-men\-sio\-nal EVD $G$ and
$\bfx_0\in\R^d$ with $G(\bfx_0)<1$ such that
\begin{equation}\label{dfn:multivariate_gpd}
Q(\bfx)=1+\ln\{G(\bfx)\},\qquad \bfx\ge \bfx_0.
\end{equation}

Note that contrary to the univariate case,
$H(\bfx)=1+\ln\{G(\bfx)\}$, defined for each $\bfx$ with
$\ln\{G(\bfx)\}\ge -1$, does \emph{not} define a distribution
function unless $d\in\set{1,2}$ (\citet[Theorem 6]{michel08}).

If $G$ has standard negative exponential margins $G_i(x)=\exp(x)$,
$x\le 0$, then $H(\bfx):=1+\ln\{G(\bfx)\}=1-\norm{\bfx}_D$,
defined for all $\bfx\le\bfzero$ with $\norm{\bfx}_D \le 1$, is a
\emph{quasi-copula} (\citet{alsns93}, \citet{genqrs99}). Note that
$H_i(x)=1+x$, $-1\le x\le 0$. We call $H$ a \emph{GP function}.
For each GP function $H$ there exists a distribution function $Q$
with $H(\bfx)=Q(\bfx)=1-\norm{\bfx}_D$, $\bfx\ge\bfx_0$, see
Corollary 2.2 in \citet{aulbf11}. We call $Q$ a
\emph{multivariate} GPD with \emph{ultimately uniform margins}.
Thus we obtain the following consequence.

\begin{cor}\label{cor:characterization_of_domain_of_attraction_for_copula}
A copula $C$ satisfies $C\in\mathcal D(G)$ if, and only if, there exists a GPD $Q$ with
ultimately uniform margins, i.e., the relation
\[
C(\bfu)=Q(\bfu-\bfone)+o(\norm{\bfu-\bfone})
\]
holds uniformly for $\bfu\in[0,1]^d$. In this  case
$Q(\bfx)=1+\ln\{G(\bfx)\}=1-\norm{\bfx}_D$, $\bfx_0\le\bfx\le\bfzero\in\R^d$.
\end{cor}

\begin{exam}
\upshape Under suitable conditions, an Archimedean copula $C_A$ is in the domain of
attraction of the EVD $G(\bfx)=\exp\left(-\norm{\bfx}_\vartheta\right)$,
$\bfx\le\bfzero\in\R^d$, where
$\norm{\bfx}_\vartheta=\left(\sum_{i=1}^d\abs{x_i}^\vartheta\right)^{1/\vartheta}$,
$\vartheta\in[1,\infty]$,  is the usual $\vartheta$-norm on $\R^d$, with the convention
$\norm{\bfx}_\infty=\max_{1\le i\le d}\abs{x_i}$; see \citet{chas09} and \citet{larn11}.
In this case it is reasonable to replace $C_A(\bfu)$ for $\bfu$ close to $\bfone$ by
$Q(\bfu-\bfone)=1-\norm{\bfu-\bfone}_\vartheta$.
\end{exam}

The multivariate PT approach in  \citet{aulbf11} is formulated in
terms of random vectors and based on the following result. Its
second part goes back to \citet{buihz08}, Section 2.2, formulated
for the bivariate case and for Pareto margins instead of uniform
ones.

\begin{lemma}\label{lem:characterization_of_multivariate_gpd}
A distribution function $Q$ is a multivariate GPD with \emph{ultimately uniform margins}
\begin{itemize}
\item[$\iff$] there exists a $D$-norm $\norm\cdot_D$ on $\R^d$ such that
    $Q(\bfx)=1-\norm\bfx_D$, $\bfx_0\le \bfx\le\bfzero\in\R^d$,
\item[$\iff$] there exists a generator $\bfZ=(Z_1,\dots,Z_d)$ such that for
    $\bfx_0\le\bfx\le\bfzero\in\R^d$
        \[
        Q(\bfx)={\rm P}\left\{-U\left(\frac 1 {Z_1},\dots,\frac1 {Z_d}\right)\le\bfx\right\},
        \]
        where the univariate random variable $U$ is uniformly distributed on $(0,1)$  and independent
        of $\bfZ$.
\end{itemize}
\end{lemma}

Note that $-U/Z_i$ can be replaced by $\max(M,-U/Z_i)$, $1\le i\le d$, in the preceding
result with some constant $M<0$ to avoid possible division by zero.

In view of the preceding discussion we call a copula $C$ a \emph{GPD-copula} if there
exists $\bfu_0<\bfone\in\R^d$ such that
\[
C(\bfu)=1-\norm{\bfu-\bfone}_D,\qquad \bfu_0\le\bfu\le\bfone\in\R^d,
\]
where $\norm\cdot_D$ is an arbitrary $D$-norm on $\R^d$, i.e., if there  exists a generator
$\bfZ=(Z_1,\dots,Z_d)$ such that for $\bfu_0\le\bfu\le\bfone\in\R^d$
        \[
        C(\bfu)={\rm P}\left\{-U\left(\frac 1 {Z_1},\dots,\frac1 {Z_d}\right)\le\bfu-\bfone\right\},
        \]
        where the random variable $U$ is uniformly distributed on $(0,1)$  and independent of $\bfZ$.

\section{Multivariate Piecing-Together}\label{sec:multivariate_piecing_together}

Let $\bfU=(U_1,\dots,U_d)$ follow an arbitrary copula $C$ and $\bfV=(V_1,\dots,V_d)$
follow a GPD-copula with generator $\bfZ$. We suppose that $\bfU$ and $\bfV$ are
independent.

Choose a threshold $\bfu=(u_1,\dots,u_d)\in(0,1)^d$ and put for $1\le i\le d$
\begin{equation}\label{eqn:definition_of_pt-vector}
Y_i:=U_i\bm{1}(U_i\le u_i) + \{u_i+(1-u_i)V_i\}\bm{1}(U_i>u_i).
\end{equation}
While it was merely shown in \citet{aulbf11} that the random vector
$\bfY=(Y_1,\dots,Y_d)$ actually follows a GPD, the following main result of this section
provides a precise characterization of the corresponding $D$-norm.

\begin{theorem}\label{theo:multivariate_piecing_together}
Suppose that ${\rm P}(\bfU>\bfu)>0$. The random vector $\bfY$ defined through
\eqref{eqn:definition_of_pt-vector} follows a GPD-copula, which coincides with $C$ on
$[\bfzero,\bfu]\in(0,1)^d$ and $D$-norm given by
\[
\norm\bfx_D={\rm E}\left[\max_{1\le j\le d}\left\{\abs{x_j}Z_j\frac{\bm{1}(U_j>u_j)}{1-u_j}\right\}\right],
\]
where $\bfZ$ and $\bfU$ are independent.
\end{theorem}

Note that $\widetilde \bfZ:=(\widetilde Z_1,\dots,\widetilde Z_d)$ with $\widetilde Z_j:= Z_j
\bm{1}(U_j>u_j)/(1-u_j)$,  is a generator with the characteristic properties of being
nonnegative, bounded and satisfying ${\rm E}(\widetilde Z_j)=1$, $1\le j\le d$, due to the
independence of $\bfZ$ and $\bfU$. In analogy to a corresponding terminology in point
process theory one might call $\widetilde\bfZ$ a \emph{thinned} generator.

\begin{proof}
Elementary computations yield
\[
{\rm P}(Y_i\le x)=x,\qquad 0\le x\le 1,
\]
i.e., $\bfY$ follows a copula. We have, moreover, for $\bfzero\le\bfx\le \bfu$
\begin{align*}
&{\rm P}(\bfY\le\bfx)\\
&=\sum_{K\subset\set{1,\dots,d}} {\rm P}\left(\bfY\le\bfx;\,U_k\le u_k,k\in K;\,U_j>u_j,j\in K^\complement\right)\\
&=\sum_{K\subset\set{1,\dots,d}} {\rm P}\Big[U_i\bm{1}(U_i\le u_i)+\{u_i+(1-u_i)V_i\}\bm{1}(U_i>u_i)\le x_i,1\le i\le d;\\
&\hspace*{5cm} U_k\le u_k,k\in K;\,U_j>u_j,j\in K^\complement \Big]\\
&={\rm P}(U_i\le x_i,\,1\le i\le d)\\
&=C(\bfx)
\end{align*}
and for $\bfu<\bfx\le\bfone$
\begin{align*}
&{\rm P}(\bfY\le\bfx)\\
&=\sum_{K\subset\set{1,\dots,d}} {\rm P}\left(\bfY\le\bfx;\,U_k\le u_k,k\in K;\,U_j>u_j,j\in K^\complement\right)\\
&=\sum_{K\subset\set{1,\dots,d}} {\rm P}\left(U_k\le u_k,k\in K;\,u_j+(1-u_j)V_j\le x_j,U_j>u_j,j\in K^\complement\right)\\
&=\sum_{K\subset\set{1,\dots,d}} {\rm P}\left(U_k\le u_k,k\in K; U_j> u_j,j\in K^\complement\right){\rm P}\left(V_j\le\frac{x_j-u_j}{1-u_j}, j\in K^\complement\right)\\
&=\sum_{K\subset\set{1,\dots,d}} {\rm E}\left[\left\{\prod_{k\in K} \bm{1}(U_k\le u_k)\right\} \left\{\prod_{j\in K^\complement} \bm{1}(U_j > u_j)\right\}\right] \\
&\hspace*{5cm} \times {\rm P}\left(V_j\le\frac{x_j-u_j}{1-u_j}, j\in K^\complement\right).
\end{align*}

If $\bfx<\bfone$  is large enough, then we have for $K^\complement\not=\emptyset$
\begin{align*}
{\rm P}\left(V_j\le\frac{x_j-u_j}{1-u_j}, j\in K^\complement\right)&= 1-{\rm E}\left\{\max_{j\in K^\complement}\left(\abs{\frac{x_j-u_j}{1-u_j}-1}Z_j\right)\right\}\\
&=1- {\rm E}\left\{\max_{j\in K^\complement}\left(\frac{\abs{x_j-1}}{1-u_j}Z_j\right)\right\}
\end{align*}
and, thus,
\begin{align*}
&{\rm P}(\bfY\le\bfx)\\
&= {\rm P}(U_k\le u_k,1\le k\le d)\\
&\quad + \sum_{K\subset\set{1,\dots,d}\atop K^\complement\not=\emptyset} {\rm E}\left[\left\{\prod_{k\in K} \bm{1}(U_k\le u_k)\right\} \left\{\prod_{j\in K^\complement} \bm{1}(U_j > u_j)\right\}\right]\\
&\hspace*{5cm}\times \left[1- {\rm E}\left\{\max_{j\in K^\complement}\left(\frac{\abs{x_j-1}}{1-u_j}Z_j\right)\right\}\right]\\
&=1-\sum_{K\subset\set{1,\dots,d}\atop K^\complement\not=\emptyset}  {\rm E}\left[\left\{\prod_{k\in K} \bm{1}(U_k\le u_k)\right\} \left\{\prod_{j\in K^\complement} \bm{1}(U_j > u_j)\right\} \max_{j\in K^\complement}\left(\frac{\abs{x_j-1}}{1-u_j}Z_j\right)\right]\\
&=1- {\rm E}\left[\sum_{K\subset\set{1,\dots,d}\atop K^\complement\not=\emptyset} \left\{\prod_{k\in K} \bm{1}(U_k\le u_k)\right\} \left\{\prod_{j\in K^\complement} \bm{1}(U_j > u_j)\right\} \max_{j\in K^\complement}\left(\frac{\abs{x_j-1}}{1-u_j}Z_j\right)\right]\\
&=1-{\rm E}\left[\max_{1\le j\le d}\left\{\abs{x_j-1}Z_j\frac{\bm{1}(U_j>u_j)}{1-u_j}\right\}\right]\\
&=1-\norm{\bfx-\bfone}_D,
\end{align*}
as we can suppose independence of $\bfU$ and the generator $\bfZ$.
\end{proof}

The following result  justifies the use of the multivariate
PT-approach as it shows that the PT vector $\bfY$, suitably
standardized, approximately follows the distribution of $\bfU$
close to one.

\begin{prop}\label{prop:justification_of_multivariate_pt-approach}
Suppose that $\bfU=(U_1,\dots,U_d)$ follows a copula $C\in\mathcal D(G)$ with
corresponding $D$-norm $\norm\cdot_D$ generated by $\bfZ$. If the random vector $\bfV$
in the definition \eqref{eqn:definition_of_pt-vector} of the PT vector $\bfY$ has this
generator $\bfZ$ as well, then we have
\[
{\rm P}(\bfU>\bfv)={\rm P}\{Y_j>u_j + v_j(1-u_j),\,1\le j\le d\mid \bfU>\bfu\} + o(\norm{\bfone-\bfv})
\]
uniformly for $\bfv\in[\bfu,\bfone]\subset\R^d$.
\end{prop}

The term $o(\norm{\bfone-\bfv})$ can be dropped in the preceding result if $C$ is a GPD-copula itself, precisely, if $C(\bfx)=1-\norm{\bfx}_D$, $\bfx\ge \bfu$.

\begin{proof}
Repeating the arguments in the proof of Corollary \ref{cor:tail_copula} we obtain
 \[
 {\rm P}(\bfU>\bfv)={\rm E}\left[\min_{1\le j\le d}\{(1-v_j)Z_j\}\right]+ o(\norm{\bfone-\bfv})
 \]
 uniformly for $\bfv\in[0,1]^d$.

 We have, on the other hand, for $\bfv$ close enough to $\bfone$
 \begin{align*}
 &{\rm P}\{Y_j>u_j+v_j(1-u_j),\,1\le j\le d\mid \bfU > \bfu\}\\
 &\qquad={\rm P}\{U<(1-v_j) Z_j,\,1\le j\le d\}
 ={\rm E}\left[\min_{1\le j\le d}\{(1-v_j)Z_j\}\right],
 \end{align*}
 which completes the proof.
 \end{proof}

If the copula $C$ is not known, the preceding PT-approach can be modified as follows, with
$C$ replaced by the \emph{empirical copula}. Suppose we are given $n$ copies
$\bfX_1,\dots,\bfX_n$ of a random vector $\bfX=(X^{(1)},\dots,X^{(d)})$. Set for $1\le
j\le d$
\[
F_{n}^{(j)}(x):=\frac 1{n+1}\sum_{i=1}^n \bm{1}(X_i^{(j)}\le x),\qquad x\in\R,
\]
which is essentially the empirical distribution function of the $j$-th components of
$\bfX_1,\dots,\bfX_n$. Transform each random vector $\bfX_i$ in the sample to the vector
of its standardized ranks $\bfR_i:=\left(F_n(X_i^{(1)}),\dots,F_n(X_i^{(d)})\right)$. The
empirical copula is then the empirical distribution function corresponding to
$\bfR_1,\dots,\bfR_n$:
\[
C_n(\bfu)=\frac 1 n\sum_{i=1}^n \bm{1}(\bfR_i\le \bfu),\qquad \bfu\in[0,1]^d.
\]
Properties of the empirical copula are well studied; we refer to \citet{segers2011} and the
literature cited therein.

Given the empirical copula $C_n$, let the random vector $\bfU^*=(U_1^*,\dots,U_d^*)$
follow this distribution function $C_n$ and let $\bfV=(V_1,\dots,V_d)$ follow a GPD-copula.
Again we suppose that $\bfU$ and $\bfV$ are independent.

Choose a threshold $\bfu=(u_1,\dots,u_d)\in(0,1)^d$ and put for $1\le i\le d$
\begin{equation}\label{eqn:piecing_together_with_empirical_copula}
Y_i^*:=U_i^*\bm{1}(U_i^*\le u_i)+ \{u_i^*+(1-u_i^*)V_i\} \bm{1}(U_i^*>u_i),
\end{equation}
where $u_i^*:={\rm P}_n(U_i^*\le u_i)$. Recall that the preceding probability is, actually,
a conditional one, given the empirical copula $C_n$. To avoid confusion we add the index
$n$. The following result can be shown by repeating the arguments in the proof of Theorem
\ref{theo:multivariate_piecing_together}. The minimum $\min(\bfu,\bfu^*)$ is meant to be
taken componentwise.

\begin{prop}
Suppose that the threshold $\bfu\in(0,1)^d$ satisfies ${\rm P}_n(\bfU^*>\bfu)>0$. The
random vector $\bfY^*$, defined componentwise in
\eqref{eqn:piecing_together_with_empirical_copula}, follows a multivariate GPD, which
coincides on $[\bfzero,\min(\bfu,\bfu^*)]$ with the empirical copula $C_n$ and, for
$\bfx<\bfone$ large enough,
\[
{\rm P}_n(\bfY^*\le\bfx)=1-\norm{\bfx}_{D_n},
\]
where the $D$-norm is given by
\[
\norm{\bfx}_{D_n}={\rm E}_n\left[\max_{1\le j\le d}\left\{\abs{x_j}Z_j\frac{\bm{1}(U_j^*>u_j)}{1-u_j^*}\right\}\right],
\]
the generator $\bfZ$ and $\bfU^*$ being independent and ${\rm E}_n$ denoting the
expected value with respect to ${\rm P}_n$.
\end{prop}

Proposition \ref{prop:justification_of_multivariate_pt-approach} can now be formulated as follows; its proof carries over.

\begin{prop}
Let $C$ be a copula with $C\in\mathcal{D}(G)$ and corresponding $D$-norm
$\norm\cdot_D$ generated by $\bfZ$. Let the random vector $\bfU$ follow this copula $C$.
Suppose that the random vector $\bfV$ in the definition
\eqref{eqn:piecing_together_with_empirical_copula} of the  PT random vector $\bfY^*$ has
this generator $\bfZ$ as well. Then we have
\[
{\rm P}(\bfU>\bfv)={\rm P}_n\{Y_j^*>u_j^*+v_j(1-u_j^*),\,1\le j\le d\mid
\bfU^*>\bfu\} + o(\norm{\bfone-\bfv})
\]
uniformly for $\bfv\in[\bfu,\bfone]\in\R^d$, where $\bfU^*$
follows the empirical copula $C_n$.
\end{prop}

The term $o(\norm{\bfone-\bfv})$ can again be dropped in the
preceding result if $C$  is a GPD-copula itself, precisely, if
$C(\bfx)=1-\norm{\bfx-\bfone}_D$, $\bfx\ge \bfu$.

\section{Piecing Together: A Functional Version}\label{sec:pt_a_functional_version}

In this section we will extend the PT approach from Section
\ref{sec:multivariate_piecing_together}  to function spaces. Suppose we are given a
stochastic process $\bfX=(X_t)_{t\in[0,1]}$ with corresponding continuous copula process
$\bfU=(U_t)_{t\in[0,1]}\in C[0,1]$, where $C[0,1]$ denotes the space of continuous
functions on $[0,1]$. A \emph{copula process} $\bfU$ is characterized by the condition
that each $U_t$ is uniformly distributed on $(0,1)$. For a review of the attempts to extend
the use of copulas to a dynamic setting, we refer to \citet{ng2010} and the review paper
by Andrew Patton in this Special Issue.

Choose a \emph{generator process} $\bfZ=(Z_t)_{t\in[0,1]}$, characterized by the condition
\[
0\le Z_t\le c,\quad {\rm E}(Z_t)=1,\qquad 0\le t\le 1,
\]
for some $c\ge 1$. We require that $\bfZ\in C[0,1]$ as well.

Let $U$ be a uniformly distributed on $(0,1)$  random variable
that is independent of $\bfZ$ and put for some $M<0$
\begin{equation} \label{eqn:definition_of_standard_GPP}
V_t:=\max\left(M,-\frac U{Z_t}\right),\qquad 0\le t\le 1.
\end{equation}
The process $\bfV=(V_t)_{t\in[0,1]}\in C[0,1]$ is called a
\emph{standard generalized Pareto process} (GPP)  as it has
ultimately uniform margins, see below. This functional extension
of multivariate GPD goes back to \citet{buihz08}, Section 2.3,
again with Pareto margins instead of uniform ones. We incorporate
the constant $M$ again to avoid possible division by zero.

Note that for $0\ge x\ge K:=\max(M,-1/c)$
\begin{align}\label{eqn:marginal_distribution_of_standard_gpp}
{\rm P}(V_t\le x)
&={\rm P}(U\ge \abs x Z_t)
=\int_0^c {\rm P}(U\ge \abs x z)\,({\rm P}*Z_t)(dz)
=1+x,
\end{align}
i.e., each $V_t$ follows close to zero a uniform distribution.

Denote by $E[0,1]$ the set of bounded functions $f:[0,1]\to\R$,
which have only a  finite number of discontinuities, and put $\bar
E^-[0,1]:=\{f\in E[0,1]:\,f\le 0\}$. Repeating the arguments in
the derivation of equation
\eqref{eqn:marginal_distribution_of_standard_gpp}, we obtain for
$f\in \bar E^-[0,1]$ with $\norm f_\infty\le \abs K$
\[
{\rm P}(\bfV\le f)={\rm P}\{V_t\le f(t),\,t\in[0,1]\}
=1- {\rm E}\left\{\sup_{t\in[0,1]}(\abs{f(t)}Z_t)\right\}.
\]
To improve the readability of this paper, we set stochastic processes such as $\bfV$ in bold font and non stochastic functions such as $f$ in default font. Operations on functions such as $\le$, $>$ etc. are meant componentwise.

The process $\bfV$ can easily be modified to obtain a \emph{generalized Pareto copula
process} (GPCP) $\bfQ=(Q_t)_{t\in[0,1]}$, i.e., each $Q_t$ follows the uniform distribution
on $(0,1)$ and $(Q_t-1)_{t\in[0,1]}$ is a GPP. Just put
\[
\widetilde V_t:=\begin{cases}
V_t&\mbox{if }V_t>K\\
\xi&\mbox{if }V_t\le K
\end{cases},\qquad 0\le t\le 1,
\]
where the random variable $\xi$ is uniformly distributed on
$(-1,K)$ and independent of the process $\bfV$; we assume that
$K>-1$. Note that each $\widetilde V_t$ is uniformly distributed
on $(-1,0)$
 and that for $f\in\bar E^-[0,1]$ with $\norm f_\infty<
\abs K$
\begin{align}
{\rm P}\left(\widetilde\bfV\le f\right)&={\rm P}\left\{\widetilde V_t\le f(t),\,0\le t\le 1\right\} \nonumber \\
&={\rm P}\{V_t\le f(t),\,0\le t\le 1\} \label{eqn:distribution_of_GPCP_upper_tail}
={\rm P}(\bfV\le f).
\end{align}
The process $\bfQ$ is now obtained by putting $\bfQ:=(\widetilde V_t+1)_{t\in[0,1]}$. It
does not have continuous sample paths, but it is continuous in probability, i.e.,
\[
{\rm P}\left(\abs{Q_{t_n}- Q_t}>\varepsilon\right)\to_{t_n\to t} 0
\]
for each $t\in[0,1]$ and any $\varepsilon>0$.

Suppose that we are given a copula process $\bfU\in C[0,1]$. Choose a GPCP $\bfQ$ with
generator $\bfZ\in C[0,1]$, $\bfQ$ independent of $\bfU$, a threshold $u\in (0,1)$ and put
\begin{equation}\label{eqn:functional_pt_version}
 Y_t:= U_t\bm{1}(U_t\le u)+\{u+(1-u)Q_t\} \bm{1}(U_t>u),\quad t\in [0,1].
\end{equation}
We call $\bfY=(Y_t)_{t\in[0,1]}$ a \emph{PT-process}. We require that the processes
$\bfU$ and $\bfZ$ are independent. Note that $\bfY$ is continuous under the condition
$\bfU>u$. The following theorem is the main result in this section.

\begin{theorem}\label{theo:pt_processes}
The process $\bfY=(Y_t)_{t\in[0,1]}$ with $Y_t$ as in
\eqref{eqn:functional_pt_version} is a GPCP, which is continuous
in probability, and with $D$-norm given by
\[
\norm f_D={\rm E}\left[\sup_{t\in[0,1]}\left\{\abs{f(t)}Z_t\frac{\bm{1}(U_t>u)}{1-u}\right\}\right],\qquad f\in E[0,1].
\]
\end{theorem}

Note that ${\rm
E}\left[\sup_{t\in[0,1]}\left\{\abs{f(t)}Z_t\bm{1}(U_t>u)/(1-u)\right\}\right]$ is well
defined, due to the continuity of $\bfZ$ and $\bfU$. The \emph{thinned} generator process
\[
\widetilde \bfZ =\left\{Z_t \frac{\bm{1}(U_t>u)}{1-u}\right\}_{t\in[0,1]}
\]
satisfies
\[
0\le \widetilde Z_t\le \frac c{1-u},\quad {\rm E}(\widetilde Z_t)=1,\qquad t\in[0,1],
\]
and it is continuous in probability.

\begin{proof}
Each $Y_t$ is by Theorem \ref{theo:multivariate_piecing_together} uniformly distributed on
$(0,1)$ . Continuity in probability follows from elementary arguments. Choose $f\in\bar
E^-[0,1]$ with $\norm f_\infty<(1-u)\min\left\{\abs M,\abs K,c^{-1}\right\}$. We have
\begin{equation*}
{\rm P}(Y_t\le 1+f(t),\,t\in[0,1])
%&={\rm P}\big[U_t\bm{1}(U_t\le u)+\{u+(1-u)Q_t\}\bm{1}(U_t>u)\le 1+f(t),\,t\in[0,1]\big]\\
={\rm P}\big[\{u+(1-u)Q_t\}\bm{1}(U_t>u)\le 1+f(t),\,t\in[0,1]\big].
\end{equation*}
Note that the term $U_t\bm{1}(U_t\le u)$ can be neglected since $U_t\le u$ implies $U_t\le
1+f(t)$ and, due to the restrictions on $f$, $1+f(t)>u > 0$, $t\in[0,1]$. Analogously, we
may rewrite the probability from above as
\begin{align*}
&{\rm P}\big\{(1-u)Q_t \bm{1}(U_t>u)\le 1-u+f(t),\,t\in[0,1]\big\}\\
%&={\rm P}\left\{Q_t \bm{1}(U_t>u)\le 1+\frac{f(t)}{1-u},\,t\in[0,1]\right\}\\
&={\rm P}\left\{(Q_t-1)\bm{1}(U_t>u)\le 1- \bm{1}(U_t>u) + \frac{f(t)}{1-u},\,t\in[0,1]\right\}\\
%&={\rm P}\left\{(Q_t-1)\bm{1}(U_t>u)\le \bm{1}(U_t\le u) + \frac{f(t)}{1-u},\,t\in[0,1]\right\}\\
&={\rm P}\left\{Q_t-1 - Q_t\bm{1}(U_t\le u)\le \frac{f(t)}{1-u}\bm{1}(U_t\le u) + \frac{f(t)}{1-u}\bm{1}(U_t> u),\,t\in[0,1]\right\}\\
&={\rm P}\left\{Q_t-1\le \frac{f(t)}{1-u} \bm{1}(U_t>u),\,t\in[0,1]\right\}
\end{align*}
where the last equality is again a consequence of neglecting the terms corresponding to the
case $U_t\le u$; note that the restrictions on $f$ imply $f(t)\ge u-1$. This probability has
by \eqref{eqn:definition_of_standard_GPP} and \eqref{eqn:distribution_of_GPCP_upper_tail}
the representation
\begin{align*}
{\rm P}\left\{V_t\le \frac{f(t)}{1-u} \bm{1}(U_t>u),\,t\in[0,1]\right\}
%&={\rm P}\left\{-U\le f(t) Z_t \frac{\bm{1}(U_t>u)}{1-u},\,t\in[0,1]\right\}\\
%&={\rm P}\left\{U\ge \abs{f(t)} Z_t \frac{\bm{1}(U_t>u)}{1-u},\,t\in[0,1]\right\}\\
&={\rm P}\left[U\ge \sup_{t\in[0,1]}\left\{\abs{f(t)} Z_t \frac{\bm{1}(U_t>u)}{1-u}\right\}\right]\\
&=1- {\rm E}\left[\sup_{t\in[0,1]}\left\{\abs{f(t)} Z_t \frac{\bm{1}(U_t>u)}{1-u}\right\}\right]
\end{align*}
which completes the proof.
\end{proof}

In what follows we justify the functional PT approach by extending Proposition \ref{prop:justification_of_multivariate_pt-approach}. We say that a copula process $\bfU\in C[0,1]$ is in the \emph{functional domain of attraction} of a max-stable process $\bfeta\in C[0,1]$, denoted by $\bfU\in\mathcal D(\bfeta)$, if
\begin{equation*}
{\rm P}\{n(\bfU-1)\le f\}^n\to_{n\to\infty}{\rm P}(\bfeta\le f),\qquad f\in\bar E^-[0,1].
\end{equation*}
The max-stability of $\bfeta$ is characterized by the equation
\[
{\rm P}\left(\bfeta\le \frac fn\right)^n={\rm P}(\bfeta\le f),\qquad n\in\N,\,f\in \bar E^-[0,1].
\]
From \citet{aulfaho11} we know that there exists a generator process $\bfZ=(Z_t)_{t\in[0,1]}\in C[0,1]$ such that for $f\in\bar E^-[0,1]$
\[
{\rm P}(\bfeta\le f)=\exp\left[-{\rm E}\left\{\sup_{t\in[0,1]}(\abs{f(t)}Z_t)\right\}\right] =\exp\left(-\norm f_D\right),
\]
which shows in particular that the process $\bfeta$ has  standard
negative exponential margins. A continuous max-stable process
(MSP) with standard negative exponential margins will be called a
\emph{standard} MSP. We refer to \citet{aulfaho11} for a detailed
investigation of the functional domain of attraction condition,
which is weaker than that based on weak convergence developed in
\citet{dehal01}.

The next result, which justifies the  functional PT-approach, is now an immediate
consequence of Proposition \ref{prop:justification_of_multivariate_pt-approach}. The term
$o(\norm{\bfone-\bfv})$ can again be dropped for $(v_1,\dots,v_d)$ large enough, if the
process $\bfU$ is itself a GPCP.

\begin{prop}
Suppose that the copula process $\bfU\in C[0,1]$ satisfies $\bfU\in\mathcal D(\bfeta)$,
where $\bfeta\in C[0,1]$ is a standard MSP with generator process
$\bfZ=(Z_t)_{t\in[0,1]}\in C[0,1]$. Choose a threshold $u\in(0,1)$ and arbitrary indices
$0\le t_1<\dots<t_d\le 1$, $d\in\N$.  If the process $\bfV$ in the definition
\eqref{eqn:functional_pt_version} of the PT-process $\bfY$ has this generator $\bfZ$ as
well, then we have
\begin{align*}
&{\rm P}\left(U_{t_j}>v_{t_j},\,1\le j\le d\right)\\
&={\rm P}\left\{Y_{t_j}>u+(1-u)v_j,\,1\le j\le d \mid U_{t_j}>u,\,1\le j\le d\right\} + o(\norm{\bfone-\bfv}),
\end{align*}
uniformly for $\bfv\in[u,1]^d$.
\end{prop}

\section*{Acknowledgments}

The first author was supported by DFG Grant FA 262/4-1. This paper has benefited
substantially from discussions during the \emph{Workshop on Copula Models and
Dependence}, June 6--9, 2011, Universit\'{e} de Mont\-r\'{e}al. The second author is in particular
grateful to Johanna Ne\v{s}lehov\'{a} and Bruno R\'{e}millard for stimulating the inclusion of the
empirical copula in the preceding PT approach.

\end{document}